\documentclass[a4paper]{article}

\usepackage{amssymb,amsmath}%,amsthm}
\usepackage{enumerate} %enumerate
\usepackage{enumitem} %format enumerate
\usepackage{graphicx}

\def\vertex{\circle*{1.8}}

\newtheorem{theorem}{Theorem}[section]
\newtheorem{lemma}[theorem]{Lemma}

\newenvironment{proof}[1][]%
{\noindent {\setcounter{equation}{0}\it Proof.
}{#1}{}}{\hfill$\Box$\vspace{2ex}}

\def\longbox#1{\parbox{0.85\textwidth}{#1}}

\begin{document}

\title{Maximum Weight Stable Set in  ($P_7$, bull)-free graphs and
($S_{1,2,3}$, bull)-free graphs}

\author{%
Fr\'ed\'eric Maffray\thanks{CNRS, Laboratoire G-SCOP, University of
Grenoble-Alpes, France.}
           \and%
Lucas Pastor\thanks{Laboratoire G-SCOP, University of Grenoble-Alpes,
France.  \newline The authors are partially supported by ANR project
STINT (reference ANR-13-BS02-0007).} }

\date{\today}

\maketitle

\begin{abstract}
We give a polynomial time algorithm that finds the maximum weight
stable set in a graph that does not contain an induced path on seven
vertices or a bull (the graph with vertices $a, b, c, d, e$ and edges
$ab, bc, cd, be, ce$). With the same arguments with also give a polynomial
algorithm for any graph that does not contain $S_{1,2,3}$ or a bull.

\medskip

\end{abstract}

\section{Introduction}
%%%%%%%%%%%%%%%%%%%%%%%%%%%%%%%%%%
%% Introduction and Definitions %%
%%%%%%%%%%%%%%%%%%%%%%%%%%%%%%%%%%

In a graph $G$, a \emph{stable set} (or \emph{independent set}) is a
subset of pairwise non-adjacent vertices.  The \emph{Maximum Stable
Set problem} (shortened as MSS) is the problem of finding a stable set
of maximum cardinality.  In the weighted version, let $w : V(G)
\rightarrow \mathbb{N}$ be the weight function over the set of
vertices.  The weight of any subset of vertices is defined as the
sum of the weight of all its elements.  The \emph{Maximum Weight
Stable Set problem} (shortened as MWSS) is the problem of finding a
stable set of maximum weight.  It is known that MSS and MWSS are
NP-hard in general \cite{GarJoh}.

Given a set of graphs $\mathcal{F}$, a graph $G$ is
\emph{$\mathcal{F}$-free} if no induced subgraph of $G$ is isomorphic
to a member of $\mathcal{F}$.  If $\mathcal{F}$ is composed of only
one element $F$, we say that $G$ is $F$-free.  On the other hand, we
say that $G$ \emph{contains} $F$ when $F$ is isomorphic to an induced
subgraph of $G$.  For any integer $k$, we let $P_k$, $C_k$ and $K_k$
denote respectively the chordless path on $k$ vertices, the chordless
cycle on $k$ vertices, and the complete graph on $k$ vertices.  The
\emph{claw} is the graph with four vertices $a,x,y,z$ and three edges
$ax,ay,az$.  Let $S_{i, j, k}$ be the graph obtained from a claw by
subdividing its edges into respectively $i$, $j$ and $k$ edges.  Let
us say that a graph is \emph{special} if every component of the graph
is a path or an $S_{i,j,k}$ for any $i,j,k$.

\begin{itemize}\itemsep0pt
\item
Alekseev \cite{A} proved that MSS remains NP-hard in the class of
$\mathcal{F}$-free graphs whenever $\mathcal{F}$ is a finite set of
graphs such that no member of $\mathcal{F}$ is special.

\item
Several authors \cite{FOS, M, NT, NS, S} proved that MWSS can be
solved in polynomial time for claw-free graphs ($S_{1, 1, 1}$-free
graphs).

\item
Lozin and Milani\v{c}~\cite{LM} proved that MWSS can be solved in
polynomial time for fork-free graphs ($S_{1, 1, 2}$-free graphs).

\item
Lokshtanov, Vatshelle and Villager~\cite{LVV} proved that MWSS can be
solved in polynomial time for $P_5$-free graphs ($S_{0, 2, 2}$-free
graphs).
\end{itemize}
The results above settle the complexity of MWSS in $F$-free graph
whenever $F$ is a connected special graph on at most five vertices.
Therefore the new frontier to explore is when the forbidden induced
subgraph has six or more vertices.  There are several results on the
existence of a polynomial time algorithm for MWSS in subclasses of
$P_6$-free graphs \cite{K, KM, MP, RM1999, RM2009, RM2013}.  Mosca
\cite{RM2008} proved that MWSS is solvable in polynomial for the class
of ($P_7$, banner)-free graphs.  Brandst\"adt and Mosca~\cite{BM}
proved that there exists a polynomial time algorithm for the MWSS
problem in the class of ($P_7$, $K_3$)-free graphs.  The \emph{bull}
is the graph with vertices $a, b, c, d, e$ and edges $ab, bc, cd, be,
ce$ (see Figure~\ref{fig:bull}).  Our main results are the following
two theorems.
\begin{theorem}\label{thm:P7}
The Maximum Weight Stable Set problem can be solved in polynomial time
in the class of ($P_7$, bull)-free graphs.
\end{theorem}

\begin{theorem}\label{thm:S123}
The Maximum Weight Stable Set problem can be solved in polynomial time
in the class of ($S_{1,2,3}$, bull)-free graphs.
\end{theorem}
Theorem~\ref{thm:P7} generalizes the main results in \cite{BM} and
\cite{MP}, and Theorem~\ref{thm:S123} generalizes the main results in
\cite{KM2} and  \cite{MP}.

\begin{figure}[ht]
\unitlength=0.08cm
\thicklines
\begin{center}
\begin{picture}(24,12) % BULL
           % vertices
\multiput(6,0)(12,0){2}{\vertex}
\multiput(0,12)(24,0){2}{\vertex}
\put(12,9){\vertex}
           % edges
\put(6,0){\line(1,0){12}}
\put(18,0){\line(1,2){6}}
\put(6,0){\line(-1,2){6}}
% \put(12,15){\line(4,1){12}} \put(12,15){\line(-4,1){12}}
\put(12,9){\line(2,-3){6}}
\put(12,9){\line(-2,-3){6}}
%
% \put(9,-4){Bull}
           % end
\end{picture}
\end{center}
\caption{The bull.}
\label{fig:bull}
\end{figure}
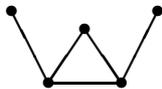

Our paper is organised as follows.  In the rest of this section we
recall some definitions, notations and well known results.  In Section
2 we develop a structural description that we can use to solve the
MWSS efficiently.  In Section~\ref{sec:P7}, thanks to the detailed
structure, we show how to solve the MWSS in polynomial time in the
class of ($P_7$, bull)-free graphs.  In Section~\ref{sec:S123}, we
show how to solve the MWSS in polynomial time in the class of
($S_{1,2,3}$, bull)-free graphs.

\medskip

Let $G$ be a graph.  For any vertex $v \in V(G)$, we denote by $N(v) =
\{u \in V(G) \mid uv \in E(G)\}$ the \emph{neighborhood} of $v$.  For
any $S \subseteq V(G)$ we denote by $G[S]$ the \emph{induced subgraph}
of $G$ with vertex-set $S$.  For any $X\subseteq V(G)$, we may write
$G\setminus X$ instead of $G[V(G)\setminus X]$.  For any $S\subseteq
V(G)$ and $x\in V(G)$, we let $N_S(x)$ stand for $N(x)\cap S$.  For
two sets $K, S \subseteq V(G)$, we say that $K$ is \emph{complete} to
$S$ if every vertex of $K$ is adjacent to every vertex of $S$, and we
say that $K$ is \emph{anticomplete} to $S$ if no vertex of $K$ is
adjacent to any vertex of $S$.  A \emph{homogeneous set} is a set $S
\subseteq V(G)$ such that every vertex in $V(G) \setminus S$ is either
complete to $S$ or anticomplete to $S$.  A homogeneous set is
\emph{proper} if it contains at least two vertices and is different
from $V(G)$.  A graph is \emph{prime} if it has no proper homogeneous
set.

\medskip

A \emph{hole} in a graph is any induced cycle on at least four
vertices.  An \emph{antihole} is the complement of a hole.  A graph
$G$ is perfect if, for every induced subgraph $G'$ of $G$, the
chromatic number of $G'$ is equal to the maximum clique size in $G'$.
The Strong Perfect Graph Theorem \cite{CRST} establishes that a graph
is perfect if and only if it contains no odd hole and no odd antihole.

\medskip

In a series of papers \cite{Ch-bull1,Ch-bull23} Chudnovsky established
a decomposition theorem for all bull-free graphs.  Based on this
decomposition, Thomass\'e, Trotignon and Vu\v{s}kovi\'c \cite{TTV}
proved that the MWSS problem is fixed-parameter tractable in the class
of bull-free graphs.  It might be that these results could be adapted
so as to yield an alternate proof of Theorems~\ref{thm:P7}
and~\ref{thm:S123}.  However we are able to avoid using the rather
complex machinery of \cite{TTV} and \cite{Ch-bull1,Ch-bull23}.  Our
proof is based on conceptually simple ideas derived from \cite{BM} and
is self-contained.

%%%%%%%%%%%%%%%%%%%%%%%%%%%%
%% Structural description %%
%%%%%%%%%%%%%%%%%%%%%%%%%%%%
\section{Structural description}

A class of graphs is \emph{hereditary} if, for every graph $G$ in the
class, every induced subgraph of $G$ is also in the class.  For
example, for any set $\mathcal{F}$ of graphs, the class of
$\mathcal{F}$-free graphs is hereditary.  We will use the following
theorem of Lozin and Milani\v{c} \cite{LM}.

\begin{theorem}[\cite{LM}]\label{thm:LM}
Let $\cal{G}$ be a hereditary class of graphs.  Suppose that there is
a constant $c \geq 1$ such that the MWSS problem can be solved in time
$O(|V(G)|^c)$ for every prime graph $G$ in $\cal{G}$.  Then the MWSS
problem can be solved in time $O(|V(G)|^c + |E(G)|)$ for every graph
$G$ in $\cal{G}$.
\end{theorem}

The classes of ($P_7$, bull)-free graphs and ($S_{1,2,3}$, bull)-free
graphs are hereditary.  Hence, in order to prove Theorems~\ref{thm:P7}
and~\ref{thm:S123} it suffices to prove them for prime graphs.

\medskip

In a graph $G$, let $H$ be a subgraph of $G$.  For each $k>0$, a
\emph{$k$-neighbor} of $H$ is any vertex in $V(G)\setminus V(H)$ that
has exactly $k$ neighbors in $H$.  The following two lemmas are
straightforward and we omit their proof.

\begin{lemma}[\cite{MP}]\label{lem:c5n}
Let $G$ be a bull-free graph.  Let $C$ be an induced $C_5$ in $G$,
with vertices $c_1, \ldots, c_5$ and edges $c_ic_{i+1}$ for each $i$
modulo $5$.  Then:
\begin{itemize}\itemsep0pt
\item
Any $2$-neighbor of $C$ is adjacent to $c_i$ and $c_{i+2}$ for some
$i$.
\item
Any $3$-neighbor of $C$ is adjacent to $c_i$, $c_{i+1}$ and $c_{i+2}$
for some $i$.
\item
If a non-neighbor of $C$ is adjacent to a $k$-neighbor of $C$, then
$k\in\{1,2,5\}$.
\end{itemize}
\end{lemma}

\begin{lemma}\label{lem:c7n}
Let $G$ be a bull-free graph.  Let $C$ be an induced $C_7$ in $G$,
with vertices $c_1, \ldots, c_7$ and edges $c_ic_{i+1}$ for each $i$
modulo $7$.  Then: \begin{itemize}\itemsep0pt
\item
Any $2$-neighbor of $C$ is adjacent to $c_i$ and either $c_{i+2}$ or
$c_{i+3}$ for some $i$.
\item
Any $3$-neighbor of $C$ is adjacent to either to $c_i$, $c_{i+1}$ and
$c_{i+2}$ or to $c_i$, $c_{i+2}$ and $c_{i+4}$ for some $i$.
\item
$C$ has no $k$-neighbor for any $k\in\{4,5,6\}$.
\end{itemize}
\end{lemma}

For any integer $k\ge 5$, a \emph{$k$-wheel} is a graph that consists
of a $C_k$ plus a vertex (called the center) adjacent to all vertices
of the cycle.  The following lemma was proved for $k\ge 7$
in~\cite{RSb}; actually the same proof holds for all $k\ge 6$ as
observed in \cite{DM,FMP}.
\begin{lemma}[\cite{RSb,DM,FMP}]\label{lem:wheel}
A prime bull-free graph contains no $k$-wheel for any $k\ge 6$.
\end{lemma}
Since the bull is a self-complementary graph, the lemma also says that
a prime bull-free graph does not contain the complementary graph of a
$k$-wheel with $k\ge 6$ (a \emph{$k$-antiwheel}).

\begin{figure}[ht]
\unitlength=0.08cm
\thicklines
\begin{center}
\begin{tabular}{cccc}
\begin{picture}(24,24) %
           % vertices
\multiput(12,0)(0,9){2}{\vertex}
\multiput(0,18)(6,0){5}{\vertex}
           % edges
\put(12,0){\line(0,1){18}}
\put(0,18){\line(1,0){24}}
\put(12,9){\line(-4,3){12}}
\put(12,9){\line(-2,3){6}}
\put(12,9){\line(2,3){6}}
\put(12,9){\line(4,3){12}}
\qbezier(0,18)(12,27)(24,18)
\put(4,-6){Umbrella}
           % end
\end{picture}
\quad & \quad
\begin{picture}(24,24) %
           % vertices
\multiput(12,0)(0,9){2}{\vertex}
\multiput(0,18)(6,0){5}{\vertex}
           % edges
\put(12,0){\line(0,1){18}}
\put(0,18){\line(1,0){24}}
\put(12,9){\line(-4,3){12}}
\put(12,9){\line(-2,3){6}}
\put(12,9){\line(2,3){6}}
\put(12,9){\line(4,3){12}}
\put(4,-6){Parasol}
           % end
\end{picture}
\quad & \quad
\begin{picture}(24,24) %
           % vertices
\multiput(6,0)(12,0){2}{\vertex}
\multiput(0,12)(24,0){2}{\vertex}
\multiput(12,18)(0,6){2}{\vertex}
\put(18,9){\vertex}
           % edges
\put(6,0){\line(1,0){12}}
\put(6,0){\line(-1,2){6}}
\put(18,0){\line(1,2){6}}
\put(0,12){\line(2,1){12}}
\put(24,12){\line(-2,1){12}}
\put(12,18){\line(0,1){6}}
\put(18,0){\line(0,1){9}}
\put(18,9){\line(2,1){6}}
\put(18,9){\line(-2,3){6}}

\put(9,-6){$G_1$}
           % end
\end{picture}
\quad & \quad
\begin{picture}(24,24) %
           % vertices
\multiput(6,0)(12,0){2}{\vertex}
\multiput(0,12)(24,0){2}{\vertex}
\multiput(12,18)(0,6){2}{\vertex}
\put(15,6){\vertex}
           % edges
\put(6,0){\line(1,0){12}}
\put(6,0){\line(-1,2){6}}
\put(18,0){\line(1,2){6}}
\put(0,12){\line(2,1){12}}
\put(24,12){\line(-2,1){12}}
\put(12,18){\line(0,1){6}}
\put(6,0){\line(3,2){9}}
\put(18,0){\line(-1,2){3}}
\put(15,6){\line(3,2){9}}
\put(9,-6){$G_2$}
           % end
\end{picture}
\end{tabular}
\end{center}
\caption{Four special graphs.}
\label{fig:4graphs}
\end{figure}
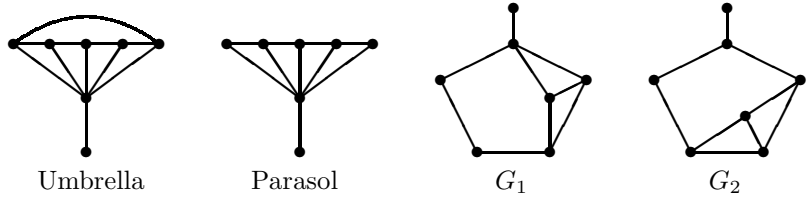

An \emph{umbrella} is a graph that consists of a $5$-wheel plus a
vertex adjacent to the center of the $5$-wheel only (see
Figure~\ref{fig:4graphs}).
\begin{lemma}[\cite{MP}]\label{lem:umbr}
A prime bull-free graph contains no umbrella.
\end{lemma}

A \emph{parasol} is a graph that consists of a $P_5$, plus a sixth
vertex adjacent to all vertices of the $P_5$, plus a seventh vertex
adjacent to the sixth vertex only (see Figure~\ref{fig:4graphs}).
\begin{lemma}\label{lem:parasol}
A prime bull-free graph contains no parasol.
\end{lemma}
\begin{proof}
Let $G$ be a prime bull-free graph, and suppose that it contains a
parasol, with vertices $p_1, \ldots, p_5, x,y$ and edges $p_ip_{i+1}$
for $i=1,2,3,4$, and $xp_j$ for $j=1,\ldots,5$ and $xy$.  Let
$P=\{p_1, \ldots, p_5\}$.  Let $A$ be the set of vertices that are
complete to $P$, and let $Z$ be the set of vertices that are
anticomplete to $P$.  Let:
\begin{eqnarray*}
A' &=& \{a\in A\mid a \mbox{ has a neighbor in } Z\}. \\
A'' &=& \{a\in A\setminus A'\mid a \mbox{ has a non-neighbor in }
A'\}.
\end{eqnarray*}
Note that $y\in Z$ and $x\in A'$, so $A'\neq\emptyset$, and that $A''$
is anticomplete to $Z$, by the definition of $A'$.  Let $H$ be the
component of $G\setminus (A'\cup A'')$ that contains $P$.  We claim
that:
\begin{equation}\label{apvh}
\mbox{$A'\cup A''$ is complete to $V(H)$.}
\end{equation}
Proof: Suppose on the contrary that there exist non-adjacent vertices
$a,u$ with $a\in A'\cup A''$ and $u\in V(H)$.  We use the following
notation.  If $a\in A'$, let $z$ be a neighbor of $a$ in $Z$.  If
$a\in A''$, let $b$ be a non-neighbor of $a$ in $A'$, and let $z$ be a
neighbor of $b$ in $Z$; in that case we know that $a$ is not adjacent
to $z$, since $a\notin A'$.  By the definition of $H$, there is a path
$u_0$-$\cdots$-$u_\ell$ in $H$ with $u_0\in P$ and $u_\ell=u$, and
$\ell\ge 0$.  We know that $a$ is adjacent to $u_0$ by the definition
of~$A$, so $\ell\ge 1$.  We choose $u$ that minimizes $\ell$, so the
path $u_0$-$\cdots$-$u_\ell$ is chordless, and $a$ is complete to
$\{u_0,\ldots,u_{\ell-1}\}$, and if $\ell\ge 2$ then
$u_2,\ldots,u_\ell\in Z$.  \\
Suppose that $\ell=1$.  Suppose that $u_1\in A$.  By the definition of
$H$ we have $u_1\in A\setminus(A'\cup A'')$, so $u_1$ is not adjacent
to $z$ and is complete to $A'$, and so $a\notin A'$, hence $a\in A''$,
and $u_1$ is adjacent to $b$.  Then $\{z,b,u_1,u_0,a\}$ induces a
bull, a contradiction.  Hence $u_1\notin A$.  So there is an integer
$i\in\{1,2,3,4\}$ such that $u_1$ has a neighbor and a non-neighbor in
$\{p_i,p_{i+1}\}$.  Suppose that $u_1$ is not adjacent to $z$.  If
$a\in A'$, then $\{z,a,p_i,p_{i+1},u_1\}$ induces a bull.  If $a\in
A''$, then $u_1$ is adjacent to $b$, for otherwise
$\{z,b,p_i,p_{i+1},u_1\}$ induces a bull; but then $\{z,b,u_1,p,a\}$
induces a bull (for $p\in\{p_i,p_{i+1}\}\cap N(u_1)$).  Hence $u_1$ is
adjacent to $z$.  It follows that there is no integer $j$ such that
$\{u_1, p_j, p_{j+1}\}$ induces a triangle, for otherwise there is an
integer $k$ such that $\{z,u_1,p_k,p_{k+1},p_{k+2}\}$ induces a bull.
If we can take $i=1$, then $u_1$ is adjacent to $p_4$, for otherwise
$\{u_1, p_1,p_2,a,p_4\}$ induces a bull; and similarly $u_1$ is
adjacent to $p_5$; but then $\{u_1, p_4, p_5\}$ induces a triangle, a
contradiction.  Hence $u_1$ is either complete or anticomplete to
$\{p_1,p_2\}$, and actually it is anticomplete to that set since
$\{u_1,p_1,p_2\}$ does not induce a triangle.  Likewise $u_1$ is
anticomplete to $\{p_4,p_5\}$.  Hence $u_1$ is adjacent to $p_3$.  But
then $\{u_1,p_3,p_2,a,p_5\}$ induces a bull, a contradiction.  \\
Therefore $\ell\ge 2$.  We have $u_1\notin A$, for otherwise we would
have $u_1\in A'$ because $u_2\in Z$.  Since $u_1\notin A$ and the
graph $\overline{P_5}$ is connected, there are non-adjacent vertices
$p,q\in P$ such that $u_1$ is adjacent to $p$ and not to $q$.  We may
assume up to relabeling that $u_0=p$.  Then $\{u_\ell, u_{\ell-1},
u_{\ell-2}, a, q\}$ induces a bull, a contradiction.  Thus
(\ref{apvh}) holds.

\medskip

Let $R=V(G)\setminus (A'\cup A''\cup V(H))$.  By the definition of
$H$, there is no edge between $V(H)$ and $R$.  By~(\ref{apvh}), $V(H)$
is complete to $A'\cup A''$.  Hence $V(H)$ is a homogeneous set, and
it is proper because $P\subseteq V(H)$ and $A'\neq \emptyset$.
\end{proof}

Let $G_1$ be the graph with vertices $p_1,\ldots,p_5,d,a$ such that
$p_1$-$p_2$-$p_3$-$p_4$-$p_5$-$p_1$ is a $C_5$, $d$ is adjacent to
$p_5$, $a$ is adjacent to $p_5,p_1,p_2$, and there is no other edge.
Let $G_2$ be the graph with vertices $p_1,\ldots,p_5,d,a$ such that
$p_1$-$p_2$-$p_3$-$p_4$-$p_5$-$p_1$ is a $C_5$, $d$ is adjacent to
$p_5$, $a$ is adjacent to $p_1,p_2,p_3$, and there is no other edge.
See Figure~\ref{fig:4graphs}.
\begin{lemma}\label{lem:nog12}
A  prime bull-free graph $G$ contains no $G_1$ and no $G_2$.
\end{lemma}
\begin{proof}
First suppose that $G$ contains a $G_1$, with the same notation as
above.  Let $X=\{x\in V(G)\mid$ $xp_5,xp_2\in E(G)$ and
$xd,xp_3,xp_4\notin E(G)\}$ (so $a,p_1\in X$), and let $Y$ be the
vertex-set of the component of $G[X]$ that contains $a$ and $p_1$.
Since $G$ is prime, $Y$ is not a homogeneous set, so there are
adjacent vertices $y,z\in Y$ and a vertex $b\in V(G)\setminus Y$ such
that $by\in E(G)$ and $bz\notin E(G)$.  Suppose that $bp_5\notin
E(G)$.  Then $bd\in E(G)$, for otherwise $\{b,y,z,p_5,d\}$ induces a
bull; and similarly $bp_4\in E(G)$.  If $bp_2\notin E(G)$, then
$bp_3\in E(G)$, for otherwise $\{b,y,z,p_2,p_3\}$ induces a bull; but
then $\{p_2,p_3,b,p_4,p_5\}$ induces a bull; so $bp_2\in E(G)$.  Then
$bp_3\in E(G)$, for otherwise $\{d,b,y,p_2,p_3\}$ induces a bull; but
then $\{d,b,p_3,p_2,z\}$ induces a bull.  Hence $bp_5\in E(G)$.
Suppose that $bp_2\notin E(G)$.  Then $bd\in E(G)$, for otherwise
$\{p_2,y,b,p_5,d\}$ induces a bull; and $bp_4\in E(G)$, for otherwise
$\{p_2,y,b,p_5,p_4\}$ induces a bull; and $bp_3\in E(G)$, for
otherwise $\{z,p_5,b,p_4,p_3\}$ induces a bull; but then
$\{d,b,p_4,p_3,p_2\}$ induces a bull.  Hence $bp_2\in E(G)$.  If
$bp_3\in E(G)$, then $bp_4\in E(G)$, for otherwise
$\{z,p_2,b,p_3,p_4\}$ induces a bull, and $bd\in E(G)$, for otherwise
$\{d,p_5,p_4,b,p_2\}$ induces a bull; but then $\{z,p_5,d,b,p_3\}$
induces a bull.  Hence $bp_3\notin E(G)$.  Then $bp_4\notin E(G)$, for
otherwise $\{p_3,p_4,b,p_5,z\}$ induces a bull, and $bd\notin E(G)$,
for otherwise $\{d,b,y,p_2,p_3\}$ induces a bull.  But now we see that
$b\in Y$, a contradiction.

\medskip

Now suppose that $G$ contains a $G_2$, with the same notation as
above.  Let $X=\{x\in V(G)\mid$ $xp_1,xp_3\in E(G)$ and
$xd,xp_5,xp_4\notin E(G)\}$ (so $a,p_2\in X$), and let $Y$ be the
vertex-set of the component of $G[X]$ that contains $a$ and $p_2$.
Since $Y$ is not a homogeneous set, there is a vertex $b \in V(G)
\setminus Y$ and two adjacent vertices $x, y \in Y$ such that $b$ is
adjacent to $x$ and not adjacent to $y$.  If $bp_4\notin E(G)$, then
$bp_3\in E(G)$, for otherwise $\{b,x,y,p_3,p_4\}$ induces a bull, and
$bp_1\in E(G)$, for otherwise $\{p_1,x,b,p_3,p_4\}$ induces a bull,
and $bp_5\notin E(G)$, for otherwise $\{y,p_1,b,p_5,p_4\}$ induces a
bull, and $bd\notin E(G)$, for otherwise $\{d,b,x,p_3,p_4\}$ induces a
bull; but then we see that $b\in Y$, a contradiction.  Hence $bp_4\in
E(G)$.  If $bp_5\in E(G)$, then $bd\in E(G)$, for otherwise $\{x, b,
p_4, p_5, d\}$ induces a bull, and $bp_3\notin E(G)$, for otherwise
$\{y,p_3,p_4,b,d\}$ induces a bull, and $bp_1\in E(G)$, for otherwise
$\{p_3,p_4,b,p_5,p_1\}$ induces a bull; but then $\{d,b,p_1,x,p_3\}$
induces a bull.  Hence $bp_5\notin E(G)$.  Then $bp_3\notin E(G)$, for
otherwise $\{y,p_3,b,p_4,p_5\}$ induces a bull, and $bp_1\in E(G)$,
for otherwise $\{b,x,y,p_1,p_5\}$ induces a bull; but then $\{p_5,
p_1,b,x,p_3\}$ induces a bull, a contradiction.
\end{proof}

%%%%%%%
\section{$(P_7,\mbox{bull})$-free graphs}
\label{sec:P7}

Before giving the proof of Theorem~\ref{thm:P7} we need another lemma.

\begin{lemma}\label{lem:c7p}
Let $G$ be a connected $(P_7,\mbox{bull})$-free graph.  Assume that
$G$ contains a $C_7$ but no $C_5$ and no $7$-wheel.  Then $V(G)$ can
be partitioned into seven non-empty sets $A_1,\ldots,A_7$ such that
for each $i\in\{1,\ldots,7\}$ $(\bmod~7)$ the set $A_i$ is complete to
$A_{i-1}\cup A_{i+1}$ and anticomplete to $A_{i-3}\cup A_{i-2}\cup
A_{i+2}\cup A_{i+3}$.
\end{lemma}
\begin{proof}
Since $G$ contains a $C_7$, there exist seven pairwise disjoint and
non-empty sets $A_1,\ldots,A_7\subset V(G)$ such that for each
$i\in\{1,\ldots,7\}$ $(\bmod~7)$ the set $A_i$ is complete to
$A_{i-1}\cup A_{i+1}$ and anticomplete to $A_{i-3}\cup A_{i-2}\cup
A_{i+2}\cup A_{i+3}$.  We choose these sets so as to maximize their
union $U=A_1\cup\cdots\cup A_7$.  Hence we need only prove that
$V(G)=U$, so suppose the contrary.  Since $G$ is connected, there is a
vertex $x$ in $V(G)\setminus U$ that has a neighbor in $U$.  For each
$i\in\{1,\ldots,7\}$ pick a vertex $c_i\in A_i$ so that $x$ has a
neighbor in the cycle $C$ induced by $\{c_1, \ldots, c_7\}$.  So $x$
is a $k$-neighbor of $C$ for some $k>0$.  Since $G$ contains no
$7$-wheel, and by Lemma~\ref{lem:c7n}, we have $k\in\{1,2,3\}$.  If
$k=1$, say $x$ is adjacent to $c_1$, then
$x$-$c_1$-$c_2$-$c_3$-$c_4$-$c_5$-$c_6$ is an induced $P_7$.  If $k=2$
and $x$ is adjacent to $c_i$ and $c_{i+3}$ for some $i$, then $\{x,
c_i, c_{i+1}, c_{i+2}, c_{i+3}\}$ induces a $C_5$.  If $k=3$ and $x$
is adjacent to $c_i$, $c_{i+2}$ and $c_{i+4}$ for some $i$, then $\{x,
c_i, c_{i-1}, c_{i-2}, c_{i-3}\}$ induces a $C_5$.  Therefore, by
Lemma~\ref{lem:c7n}, it must be that $N_C(x)$ is equal to either
$\{c_{i-1}, c_{i+1}\}$ or $\{c_{i-1}, c_i, c_{i+1}\}$ for some $i$,
say $i=7$.  Pick any $c'\in A_1\setminus \{c_1\}$ and let $C'$ be the
cycle induced by $(V(C)\setminus\{c_1\})\cup\{c'\}$.  Then by the same
arguments applied to $C'$ and $x$, we deduce that $x$ is adjacent to
$c'$.  So $x$ is complete to $A_1$, and similarly $x$ is complete to
$A_6$.  Likewise, Lemma~\ref{lem:c7n} and the fact that $G$ is
$C_5$-free implies that $x$ has no neighbor in $A_2\cup A_3\cup
A_4\cup A_5$.  But now the sets $A_1,\ldots,A_6,A_7\cup\{x\}$
contradict the maximality of $U$.  So $V(G)=U$ and the lemma holds.
\end{proof}

Now we can prove the main result of this section.

\medskip

\noindent{\bf Proof of Theorem~\ref{thm:P7}.} Let $G$ be a
$(P_7,\mbox{bull})$-free graph, and let $w$ be a weight function on
the vertex set of $G$.  By Theorem~\ref{thm:LM}, we may assume that
$G$ is prime.  By Lemmas~\ref{lem:wheel}---\ref{lem:nog12}, $G$
contains no $k$-wheel and no $k$-antiwheel for any $k\ge 6$, no
umbrella, no parasol, no $G_1$ and no $G_2$.  To find the maximum
weight stable set in $G$ it is sufficient to compute, for every vertex
$c$ of $G$, a maximum weight stable set containing $c$, and to choose
the best set over all $c$.  So let $c$ be any vertex in $G$.  The
maximum weight of a stable set that contains $c$ is equal to
$w(c)+\sum_{K} \alpha_w(K)$, where the sum is over all components $K$
of $G\setminus(\{c\}\cup N(c))$ (the non-neighborhood of $c$) and
$\alpha_w(K)$ is the maximum weight of any stable set in $K$.  So let
$K$ be an arbitrary component of $G\setminus(\{c\}\cup N(c))$.  If $K$
is perfect, we can use the algorithm from~\cite{Penev} to find a
maximum weight stable set in $K$.  Therefore let us assume that $K$ is
not perfect.  We note that $K$ contains no antihole of length at least
$6$, for otherwise the union of such a subgraph with $c$ forms an
antiwheel.  Hence, by the Strong Perfect Graph Theorem \cite{CRST},
and since $G$ is $P_7$-free, $K$ contains a $C_5$ or a $C_7$.

Since $G$ is prime it is connected, so there is a neighbor $d$ of $c$
that has a neighbor in $K$.  Let $H=N_K(d)$ and $Z=V(K)\setminus H$.
We claim that every $C_5$ in $K$ contains at most two vertices from
$H$, and if it contains two they are non-adjacent.  Indeed, in the
opposite case, there is a $C_5$ in $K$ with vertices $v_1,\ldots,v_5$
and edges $v_iv_{i+1}$ ($\bmod~5$) such that $v_1,v_2\in H$.  Then
$v_3\in H$, for otherwise $\{c,d,v_1,v_2,v_3\}$ induces a bull; and
similarly $v_4,v_5\in H$; but then $\{v_1,\ldots,v_5,d,c\}$ induces an
umbrella, which contradicts Lemma~\ref{lem:umbr}.  So the claim is
established.  Henceforth, for $q\in\{0,1,2\}$ we say that a $C_5$ in
$K$ is of type~$q$ if it contains exactly $q$ vertices from $H$.  So
every $C_5$ in $K$ is of type~$0$, $1$ or $2$, and if it is of
type~$2$ its two vertices from $H$ are non-adjacent.  Our proof
follows the pattern from \cite{BM}, but in some parts we will use
different arguments.

\bigskip

\noindent{\bf Case 1: $K$ contains a $C_7$ and no $C_5$.}

Since $K$ is connected and contains no $7$-wheel, Lemma~\ref{lem:c7p}
implies that $V(K)$ can be partitioned into seven non-empty sets
$A_1,\ldots,A_7$ such that for each $i\in\{1,\ldots,7\}$ ($\bmod 7$)
the set $A_i$ is complete to $A_{i-1}\cup A_{i+1}$ and anticomplete to
$A_{i-3}\cup A_{i-2}\cup A_{i+2}\cup A_{i+3}$.  Clearly we have
$\alpha_w(K)=\max_{i\in\{1,\ldots,7\}} \{\alpha_w(G[A_i])
+\alpha_w(G[A_{i+2}]) +\alpha_w(G[A_{i+4}])\}$, so we need only
compute $\alpha_w(G[A_i])$ for each $i\in\{1,\ldots,7\}$.  For each
$i$ pick a vertex $a_i\in A_i$.  The graph $G[A_i]$ contains no $C_5$,
no $P_5$ and no $\overline{P_5}$, for otherwise adding $a_{i+1}$ and
either $a_{i+2}$ or $a_{i+3}$ to such a subgraph we obtain an umbrella
or a parasol in $G$ or $\overline{G}$, which contradicts
Lemmas~\ref{lem:umbr} and~\ref{lem:parasol}.  By results from
\cite{CHMW} and \cite{HoangPO}, MWSS can be solved in time $O(n^3)$ in
graphs with no $C_5$, $P_5$ and $\overline{P_5}$.  Hence, since the
$A_i$'s are pairwise disjoint, MWSS can be solved in time
$O(|V(K)|^3)$ in $K$.

\bigskip

\noindent{\bf Case 2: $K$ contains a $C_5$ of type~$2$ and no $C_5$ of
type~$1$ or~$0$.}

For adjacent vertices $u,v$ in $Z$ we say that the edge $uv$ is
\emph{red} if there exists a $P_4$ $h'$-$u$-$v$-$h''$ for some
$h',h''\in H$.  For every vertex $h$ in $H$ we define its
\emph{score}, $sc(h)$, as the number of red edges that contain a
neighbor of $h$.  Let $h$ be a vertex of maximum score in $H$.
\begin{equation}\label{xmax0}
\longbox{Suppose that $K\setminus N(h)$ contains a $C_5$ of type~$2$
$t$-$h_1$-$a$-$b$-$h_2$-$t$, with $h_1,h_2\in H$ and $a,b,t\in Z$.
Then $hh_1, hh_2\notin E(G)$, and $Z$ contains vertices
$y_1,z_1,y_2,z_2$ such that $y_1z_1, y_2z_2, hy_1,hy_2\in E(G)$,
$hz_1,hz_2, h_1y_1,h_1z_1, h_2y_2,h_2z_2\notin E(G)$, and, up to
symmetry, $\{y_1,y_2\}$ is complete to~$a$ and anticomplete to~$b$,
and $\{z_1,z_2\}$ is anticomplete to $a$, and $bz_2\in E(G)$.}
\end{equation}
Proof: Clearly $h\notin\{h_1,h_2\}$.  Note that $ab$ is a red edge.
There must be a red edge $y_1z_1$ (with $y_1,z_1\in Z$) that is
counted in $sc(h)$ and not in $sc(h_1)$, for otherwise we have
$sc(h_1)\ge sc(h)+1$ (because of $ab$), which contradicts the choice
of $h$.  So $h_1$ has no neighbor in $\{y_1,z_1\}$.  We may assume
that $hy_1\in E(G)$.  Let $h'$-$y_1$-$z_1$-$h''$ be a $P_4$ with
$h',h''\in H$.  If $hz_1\in E(G)$, then $hh'\notin E(G)$, for
otherwise $\{c,d,h',h,z_1\}$ induces a bull; and similarly $hh''\notin
E(G)$; but then $\{h',y_1,h,z_1,h''\}$ induces a bull.  Hence
$hz_1\notin E(G)$.  Clearly $a\notin\{y_1,z_1\}$.  If $a$ has no
neighbor in $\{y_1,z_1\}$, then $b$ has a neighbor in $\{y_1,z_1\}$,
for otherwise $b$-$a$-$h_1$-$d$-$h$-$y_1$-$z_1$ is an induced $P_7$;
and $b$ is adjacent to both $y_1,z_1$, for otherwise
$\{c,d,h_1,a,b,y_1,z_1\}$ induces a $P_7$; but then
$\{h,y_1,z_1,b,a\}$ induces a bull, a contradiction.  So $a$ has a
neighbor in $\{y_1,z_1\}$.  If $a$ is adjacent to both $y_1,z_1$, then
$\{h,y_1,z_1,a,h_1\}$ induces a bull.  So $a$ has exactly one neighbor
in $\{y_1,z_1\}$, which leads to the following two cases: \\
--- (i) $ay_1\in E(G)$ and $az_1\notin E(G)$.  Then also $y_1b\notin
E(G)$, for otherwise $\{h,y_1,b,a,h_1\}$ induces a bull.  \\
--- (ii) $az_1\in E(G)$ and $ay_1\notin E(G)$.  Then also $z_1b\notin
E(G)$, for otherwise either $\{h_1,a,z_1,b,h_2\}$ induces a bull (if
$z_1h_2\notin E(G)$), or $\{c,d,h_1,a,z_1,b,h_2\}$ induces a $G_2$ (if
$z_1h_2\in E(G)$), which contradicts Lemma~\ref{lem:nog12}.  Moreover,
$y_1b\in E(G)$, for otherwise $c$-$d$-$h$-$y_1$-$z_1$-$a$-$b$ is an
induced $P_7$.  \\
Similarly, there is a red edge $y_2z_2$ (with $y_2,z_2\in Z$) that is
counted in $sc(h)$ and not in $sc(h_2)$, so $h_2$ has no neighbor in
$\{y_2,z_2\}$.  We may assume that $hy_2\in E(G)$, and by the same
argument as above we have $hz_2\notin E(G)$ and either: \\
--- (iii) $by_2\in E(G)$, $bz_2\notin E(G)$, and $y_2a\notin E(G)$, or \\
--- (iv) $bz_2\in E(G)$, $by_2\notin E(G)$, $z_2a\notin E(G)$, and
$y_2a\in E(G)$.  \\
Now if either (i) and (iii) occur, or (ii) and (iv) occur, then either
$\{d,h,y_1,y_2,a\}$ induces a bull (if $y_1y_2\in E(G)$) or
$\{h,y_1,y_2,a,b\}$ induces a $C_5$ of type~$1$ (if $y_1y_2\notin
E(G)$), a contradiction.  Therefore we may assume, up to symmetry,
that (i) and (iv) occur.  Thus (\ref{xmax0}) holds.

\medskip

Now we claim that:
\begin{equation}\label{xmax}
\longbox{If $K\setminus N(h)$ contains a $C_5$ of type~$2$, with the
same notation as in (\ref{xmax0}), then $K\setminus (N(h)\cup N(a))$
contains no $C_5$ of type~$2$.}
\end{equation}
Proof: Let $y_1,z_1,y_2,z_2$ be vertices of $Z$ as in (\ref{xmax0}).
Suppose that $K\setminus (N(h)\cup N(a))$ contains a $C_5$ of type~$2$
$t'$-$h_3$-$a'$-$b'$-$h_4$-$t'$, with $h_3,h_4\in H$ and $t',a',b'\in
Z$.  By the analogue of (\ref{xmax0}) there exist vertices $y_4,z_4$
in $Z$ such that $y_4z_4, hy_4\in E(G)$, $hz_4, h_4y_4,h_4z_4\notin
E(G)$, and, up to symmetry, $y_4a', z_4b'\in E(G)$ and
$y_4b',z_4a'\notin E(G)$.  We have $y_4a\notin E(G)$, for otherwise
$c$-$d$-$h_4$-$b'$-$a'$-$y_4$-$a$ is an induced $P_7$; and
$y_4y_1\notin E(G)$, for otherwise $\{d,h,y_4,y_1,a\}$ induces a bull;
and $y_4b\notin E(G)$, for otherwise $\{h,y_1,a,b,y_4\}$ induces a
$C_5$ of type~$1$.  Then $ba'\notin E(G)$, for otherwise
$c$-$d$-$h$-$y_4$-$a'$-$b$-$a$ is an induced $P_7$.  If $y_1b'\in
E(G)$, then $y_1z_4\in E(G)$, for otherwise $\{h,y_1,b',z_4,y_4\}$
induces a $C_5$ of type~$1$, and $y_1h_4\in E(G)$, for otherwise
$\{h,y_1,z_4,b',h_4\}$ induces a bull; but then $\{d,h_4,b',y_1,a\}$
induces a bull.  So $y_1b'\notin E(G)$.  Then $bb'\notin E(G)$, for
otherwise $c$-$d$-$h$-$y_1$-$a$-$b$-$b'$ is an induced $P_7$.  Then
$a'y_1\in E(G)$, for otherwise $b$-$a$-$y_1$-$h$-$y_4$-$a'$-$b'$ is an
induced $P_7$.  Then $h_3y_1\notin E(G)$, for otherwise
$\{d,h_3,a',y_1,a\}$ induces a bull, and $h_3b\notin E(G)$, for
otherwise $\{h_3,a',y_1,a,b\}$ induces a $C_5$ of type~$1$.  But then
$c$-$d$-$h_3$-$a'$-$y_1$-$a$-$b$ is an induced $P_7$, a contradiction.
Thus (\ref{xmax}) holds.

\bigskip

\noindent{\bf Case 3: $K$ contains a $C_5$ of type~$0$ or~$1$.}

We will prove that:
\begin{equation}\label{case2}
\longbox{There is a vertex $x\in V(K)$ such that $K\setminus N(x)$
contains no $C_5$ of type~$0$ or~$1$.}
\end{equation}

We first make some remarks about the $C_5$'s of type~$1$ and make a
few more claims.  Let $H_1=\{h\in H\mid$ $h$ lies in a $C_5$ of
type~$1\}$.
\begin{equation}\label{a1234}
\longbox{Let $h\in H_1$, and let $C = h$-$p_1$-$p_2$-$p_3$-$p_4$-$h$
be any $C_5$ of type~$1$ that contains~$h$.  Let $a$ be any vertex in
$Z$.  Then either $N_C(a)$ is a stable set, or $N_C(a)=
\{p_1,p_2,p_3,p_4\}$.}
\end{equation}
Proof: Suppose that $N_C(a)$ is not a stable set.  If $a$ is adjacent
to $h$ and one of $p_1,p_4$, say $ap_1\in E(G)$, then $ap_2\in E(G)$,
for otherwise $\{d,h,a,p_1,p_2\}$ induces a bull, and $ap_3\notin
E(G)$, for otherwise $\{d,h,p_1,a,p_3\}$ induces a bull, and
$ap_4\notin E(G)$, for otherwise $\{d,h,p_4,a,p_2\}$ induces a bull.
But then $\{p_1,p_2,p_3,p_4,h,d,a\}$ induces a $G_1$, which
contradicts Lemma~\ref{lem:nog12}.  Now suppose that $ah\notin E(G)$.
If $a$ is adjacent to $p_2$ and $p_3$, then $a$ also has a neighbor in
$\{p_1,p_4\}$, for otherwise $\{p_1,p_2,a,p_3,p_4\}$ induces a bull.
So in any case, up to symmetry, we may assume that $a$ is adjacent to
$p_1$ and $p_2$.  Then $ap_3\in E(G)$, for otherwise
$\{h,p_1,a,p_2,p_3\}$ induces a bull, and $ap_4\in E(G)$, for
otherwise $\{p_1,p_2,p_3,p_4,h,d,a\}$ induces a $G_2$, which
contradicts Lemma~\ref{lem:nog12}.  Thus (\ref{a1234}) holds.

\begin{equation}\label{match}
\longbox{Let $h\in H_1$, and let $C = h$-$t$-$u$-$v$-$w$-$h$ be any
$C_5$ of type~$1$ that contains~$h$.  Suppose that
$C'=h'$-$t'$-$u'$-$v'$-$w'$-$h'$ is a $C_5$ of type~$1$ in which $h$
has no neighbor, with $h'\in H$.  Then either $N_{C'}(t)=\{u',w'\}$
and $N_{C'}(w)=\{t',v'\}$, or vice-versa.}
\end{equation}
Proof: Clearly $h\neq h'$.  Let $Y=\{t,u,v,w\}$ and
$Y'=\{t',u',v',w'\}$.  Suppose that $\{t,w\}$ is anticomplete to $Y'$.
Then $h'w\in E(G)$, for otherwise $w$-$h$-$d$-$h'$-$w'$-$v'$-$u'$ is
an induced $P_7$, and similarly $h't\in E(G)$.  If $h'u\in E(G)$, then
$ut'\notin E(G)$ (by (\ref{a1234}) applied to $C'$ and $u$), but then
$\{h,t,u,h',t'\}$ induces a bull.  So $h'u\notin E(G)$, and similarly
$h'v\notin E(G)$.  Then one of $u,v$, say $u$, has a neighbor in $Y'$,
for otherwise $u$-$v$-$w$-$h'$-$w'$-$v'$-$u'$ is an induced $P_7$;
moreover $u$ is complete to $Y'$, for otherwise $c,d,h,t,u$ plus two
vertices from $Y'$ induce a $P_7$.  Then $v$ has no neighbor $y'\in
Y'$, for otherwise $\{t,u,y',v,w\}$ induces a bull; but then
$\{h',t',u',u,v\}$ induces a bull.  So $\{t,w\}$ is not anticomplete
to $Y'$, and we may assume up to symmetry that $w$ has a neighbor in
$Y'$.  \\
We have $|N_{Y'}(w)|\ge 2$ and $N_{Y'}(w)\neq\{t',w'\}$, for otherwise
$c,d,h,w$ plus three vertices from $Y'$ induce a $P_7$; and $w$ is not
complete to $Y'$, for otherwise, by (\ref{a1234}), $\{h,w,v',w',h'\}$
induces a bull.  Hence, by (\ref{a1234}) and up to symmetry, we have
$N_{C'}(w)=\{t',v'\}$.  Since $t'h\notin E(G)$, we have $t'v\notin
E(G)$, for otherwise, by (\ref{a1234}), $\{h,w,v,t',h'\}$ induces a
bull.  If also $t$ has a neighbor in $Y'$, then by the same argument
as with $w$ we have either (i) $N_{C'}(t)=\{u',w'\}$ or (ii)
$N_{C'}(t)=\{t',v'\}$.  In case (i) we obtain the desired result, so
assume that (ii) holds.  By (\ref{a1234}), $t'u\notin E(G)$.  Then
$h'$ has a neighbor in $\{u,v\}$, for otherwise
$c$-$d$-$h'$-$t'$-$w$-$v$-$u$ is an induced $P_7$; say $h'u\in E(G)$.
Then $h'v\notin E(G)$, for otherwise $\{t,u,h',v,w\}$ induces a bull.
Then $v'$ has neighbor in $\{u,v\}$, for otherwise
$c$-$d$-$h'$-$u$-$v$-$w$-$v'$ is an induced $P_7$; and by
(\ref{a1234}) we have $N_C(v')=Y$.  But then $\{h,t,v',u,h'\}$ induces
a bull, a contradiction.  So we may assume that $t$ has no neighbor in
$Y'$.  Then $th'\in E(G)$, for otherwise
$t$-$h$-$d$-$h'$-$t'$-$u'$-$v'$ is an induced $P_7$; and $uh'\notin
E(G)$, for otherwise by (\ref{a1234}), $N_C(h') = Y$, which would
imply $N_{C'}(w) \neq \{v', t'\}$; and $vh'\in E(G)$, for otherwise
$c$-$d$-$h'$-$t$-$u$-$v$-$w$ is an induced $P_7$.  By (\ref{a1234}) we
have $|N_{Y'}(v)|\le 1$ and $N_{Y'}(v)\subset\{u',v'\}$.  We have $vv'
\notin E(G)$, for otherwise $\{h, w, v, v', u'\}$ induces a bull, so
we have $vu'\in E(G)$, for otherwise $c$-$d$-$h'$-$v$-$w$-$v'$-$u'$ is
an induced $P_7$.  Then $uu'\notin E(G)$ by (\ref{a1234}) (since
$wu'\notin E(G)$).  But then $c$-$d$-$h$-$t$-$u$-$v$-$u'$ is an
induced $P_7$.  Thus (\ref{match}) holds.

\medskip

Now we deal with $C_5$'s of type~$0$.  Clearly any such $C_5$ lies in
a component of $G[Z]$, and any such component has a neighbor in $H$
since $G$ is connected.
\begin{equation}\label{z0}
\longbox{Let $T$ be any component of $G[Z]$ that contains a $C_5$, let
$C$ be any $C_5$ in $T$, and let $h$ be any vertex in $H$ that has
a neighbor in $T$.  Then $h$ is a $2$-neighbor of $C$.}
\end{equation}
Proof: There is a shortest path $p_0$-$p_1$-$p_2$-$\cdots$-$p_r$ such
that $p_0 = c$, $p_1 = d$, $p_2 = h$ and $p_r$ has a neighbor in $C$,
and $r \geq 2$.  By Lemma~\ref{lem:c5n}, $p_r$ is either a
$1$-neighbor, a $2$-neighbor or a $5$-neighbor of $C$.  If $p_r$ is a
$5$-neighbor, then $V(C)\cup\{p_r,p_{r-1}\}$ induces an umbrella,
which contradicts Lemma~\ref{lem:umbr}.  If $p_r$ is a $1$-neighbor of
$C$, then $p_{r-2}, p_{r-1}, p_r$ and four vertices of $C$ induce a
$P_7$.  So $p_r$ is a $2$-neighbor of $C$.  Now if $r\ge 3$, then
$p_{r-3}, p_{r-2}, p_{r-1}, p_r$ and three vertices of $C$ induce a
$P_7$.  So $r=2$, and (\ref{z0}) holds.

\begin{equation}\label{z01}
\mbox{At most one component of $G[Z]$ contains a $C_5$.}
\end{equation}
Proof: Suppose that two components $T$ and $T'$ of $G[Z]$ contain a
$C_5$.  Let $C$ a $C_5$ in $T$, with vertices $c_1, \ldots, c_5$ and
edges $c_ic_{i+1}$ ($\bmod~5$), and let $C'$ a $C_5$ in $T'$, with
vertices $c'_1, \ldots, c'_5$ and edges $c'_ic'_{i+1}$ ($\bmod~5$).
Pick any $h \in H$ that has a neighbor in $T$, and pick any $h'$ in
$H$ that has a neighbor in $T'$.  By (\ref{z0}) and
Lemma~\ref{lem:c5n} we may assume that $N_C(h)=\{c_1,c_4\}$ and
$N_{C'}(h')=\{c'_1,c'_4\}$.  If $h$ has a neighbor in $T'$, then, by
(\ref{z0}) and Lemma~\ref{lem:c5n}, we have $N_{C'}(h)
=\{c'_j,c'_{j+2}\}$ for some $j$.  But then
$c_3$-$c_2$-$c_1$-$h$-$c'_j$-$c'_{j-1}$-$c'_{j-2}$ is an induced
$P_7$.  So $h$ has no neighbor in $T'$, and similarly $h'$ has no
neighbor in $T$.  Then either $c_3$-$c_2$-$c_1$-$h$-$d$-$h'$-$c'_1$ or
$c_3$-$c_2$-$c_1$-$h$-$h'$-$c'_1$-$c'_2$ is an induced $P_7$.  So
(\ref{z01}) holds.

\begin{equation}\label{hinz0}
\longbox{If a component $T$ of $G[Z]$ contains a $C_5$, and $h$ is any
vertex in $H$ that has a neighbor in $T$, then $K\setminus N(h)$ has
no $C_5$ of type~$0$ or~$1$.}
\end{equation}
Proof: By (\ref{z0}) and~(\ref{z01}), $K\setminus N(h)$ has no $C_5$
of type~$0$.  So suppose that there is a $C_5$ of type~$1$
$C'=h'$-$t'$-$u'$-$v'$-$w'$-$h'$ (with $h'\in H$) in which $h$ has no
neighbor.  Let $C$ be a $C_5$ in $T$, with vertices $c_1,\ldots,c_5$
and edges $c_ic_{i+1}$ ($\bmod~5$).  By (\ref{z0}) and
Lemma~\ref{lem:c5n}, we may assume that $N_C(h)=\{c_1,c_4\}$.  Let
$C_h=h$-$c_1$-$c_2$-$c_3$-$c_4$-$h$; so $C_h$ is a $C_5$ of type~$1$.
By (\ref{match}) and up to symmetry, we have $N_{C'}(c_1)=\{t',v'\}$
and $N_{C'}(c_4)=\{u',w'\}$, and $t',u',v',w'\in T$.  Then $c_5$ has a
neighbor in $\{u',v'\}$, for otherwise $\{c_1,v',u',c_4,c_5\}$ induces
a $C_5$ of type~$0$ in which $h'$ has at most one neighbor,
contradicting (\ref{z0}).  If $c_5u' \in E(G)$, then $c_5v' \in E(G)$,
for otherwise $\{h, c_4, c_5, u', v'\}$ induces a bull.  If $c_5v' \in
E(G)$, then $c_5u' \in E(G)$, for otherwise $\{h, c_1, c_5, v', u'\}$
induces a bull.  In both cases, by (\ref{a1234}), $c_5$ is complete to
$\{t', u', v', w'\}$.  But then $\{h,c_1,t',c_5,w'\}$ induces a bull.
Thus (\ref{hinz0}) holds.

\begin{equation}\label{u}
\longbox{Suppose that there is no $C_5$ of type~$0$.  Pick any $h\in
H_1$, and suppose that there is a $C_5$ of type~$1$
$C'=h'$-$b_2$-$u$-$v$-$a_2$-$h'$ in which $h$ has no neighbor.  Then
$K\setminus N(u)$ has no $C_5$ of type~$1$.}
\end{equation}
Proof: Let $h$-$a_1$-$v'$-$u'$-$b_1$-$h$ be any $C_5$ of type~$1$ that
contains $h$.  By (\ref{match}), we may assume that
$N_{C'}(a_1)=\{b_2,v\}$ and $N_{C'}(b_1)=\{a_2,u\}$.  Let
$C=h$-$a_1$-$v$-$u$-$b_1$-$h$; then $C$ is a $C_5$ of type~$1$ in
which $h'$ has no neighbor, so $h$ and $h'$ play symmetric roles.  Let
$C_{a_1}= h$-$a_1$-$b_2$-$u$-$b_1$-$h$ and $C_{a_2}=
h'$-$a_2$-$b_1$-$u$-$b_2$-$h'$.  Suppose that there is a $C_5$ of
type~$1$ $C''=h''$-$t''$-$u''$-$v''$-$w''$-$h''$ in which $u$ has no
neighbor.  Let $X=\{a_1,b_1,a_2,b_2,u,v\}$ and
$Y''=\{t'',u'',v'',w''\}$. \\
We observe that $G[X\cup Y'']$ is bipartite: indeed in the opposite
case, and since $K$ contains no $C_5$ of type~$0$ and no $C_7$, there
is a triangle in $G[X\cup Y'']$, and so there is either (i) a vertex
$y''\in Y''$ with two adjacent neighbors in $X$, or (ii) a vertex
$x\in X$ with two adjacent neighbors in $Y''$.  In case~(i), by
(\ref{a1234}) applied to $y''$ and the cycles $C, C', C_{a_1},
C_{a_2}$, we see that $y''$ is complete to $X$, which is not possible
since $uy''\notin E(G)$.  So suppose we have case~(ii).  By
(\ref{a1234}) we have $N_{C''}(x)=Y''$.  Clearly $x\neq u$.  Moreover,
$x\notin\{b_1,b_2,v\}$, for otherwise $\{u,x,v'',w'',h''\}$ induces a
bull.  So, up to symmetry, $x=a_1$.  By case (i) we have
$v''b_2,w''b_2\notin E(G)$; but then $\{h'',w'',v'',a_1,b_2\}$ induces
a bull.  So $G[X\cup Y'']$ is bipartite.  Let $A,B$ be a bipartition
of $X\cup Y''$ in two stable sets.  Up to symmetry we may assume that
$A=\{a_1,a_2,u,u'',w''\}$ and $B=\{b_1,b_2,v,t'',v''\}$.  \\
Note that $h''$ has a neighbor in $C$, for otherwise (\ref{match}) is
contradicted (since $u$ has no neighbor in $\{t'',w''\}$), and
similarly $h''$ has a neighbor in $C'$, in $C_{a_1}$ and in $C_{a_2}$.
Suppose that $h''a_1\in E(G)$.  Then $h''b_2\notin E(G)$, for
otherwise $\{d,h'',a_1,b_2,u\}$ induces a bull, and $h''b_1\in E(G)$,
for otherwise $c$-$d$-$h''$-$a_1$-$b_2$-$u$-$b_1$ is an induced $P_7$,
and $h''a_2\notin E(G)$, for otherwise $\{d,h'',a_2,b_1,u\}$ induces a
bull, and $h''h'\notin E(G)$, for otherwise $\{c,d,h'',h',a_2\}$
induces a bull.  By (\ref{a1234}), $h''$ is not adjacent to $v$.  But
then $h''$ has no neighbor in $C'$, a contradiction.  So $h''a_1\notin
E(G)$, and similarly $h''a_2\notin E(G)$.  So $h''\notin\{h,h'\}$;
moreover $h''h\notin E(G)$, for otherwise $\{c,d,h'',h,a_1\}$ induces
a bull, and similarly $h''h'\notin E(G)$.  Then $h''$ has a neighbor
in $\{b_1,b_2\}$, say $h''b_1\in E(G)$, for otherwise $h''$ has no
neighbor in $C_{a_1}$; and then $h''b_2\in E(G)$, for otherwise
$c$-$d$-$h''$-$b_1$-$u$-$b_2$-$a_1$ is an induced $P_7$, and $h''v\in
E(G)$, for otherwise $c$-$d$-$h''$-$b_1$-$a_2$-$v$-$a_1$ is an induced
$P_7$.  So $N_X(h'')=\{b_1,b_2,v\}$.  By (\ref{a1234}), $b_1,b_2$ and
$v$ have no neighbor in $\{t'',w''\}$; and since $B$ is a stable set
they are not adjacent to $v''$.  \\
Suppose that $b_1u''\in E(G)$.  Then $a_1v''\notin E(G)$, for
otherwise $c$-$d$-$h''$-$b_1$-$u''$-$v''$-$a_1$ is an induced $P_7$,
and $hu''\notin E(G)$, for otherwise $\{d,h,u'',b_1,u\}$ induces a
bull.  Then $h$ has exactly one neighbor in $\{v'',w''\}$, for
otherwise either $c$-$d$-$h$-$b_1$-$u''$-$v''$-$w''$ is an induced
$P_7$ or $\{d,h,w'',v'',u''\}$ induces a bull.  However, if $hw''\in
E(G)$, then $b_2u''\in E(G)$, for otherwise
$u''$-$v''$-$w''$-$h$-$a_1$-$b_2$-$u$ is an induced $P_7$, and then
$c$-$d$-$h$-$w''$-$v''$-$u''$-$b_2$ is an induced $P_7$; while if
$hv''\in E(G)$, then $u''v\notin E(G)$, for otherwise
$c$-$d$-$h$-$v''$-$u''$-$v$-$u$ is an induced $P_7$, and then
$u''$-$v''$-$h$-$d$-$h''$-$v$-$u$ is an induced $P_7$, a
contradiction.  Hence $b_1u''\notin E(G)$ and, by symmetry, $b_1$ and
$b_2$ have no neighbor in $Y''$.  \\
If $vu''\in E(G)$, then $hu''\in E(G)$, for otherwise
$c$-$d$-$h$-$b_1$-$u$-$v$-$u''$ is an induced $P_7$, but then
$c$-$d$-$h$-$u''$-$v$-$u$-$b_2$ is an induced $P_7$.  So $vu''\notin
E(G)$, and so $v$ has no neighbor in $Y''$.  Then $a_1v''\notin E(G)$,
for otherwise $c$-$d$-$h''$-$v$-$a_1$-$v''$-$u''$ is an induced $P_7$;
and $a_1t''\notin E(G)$, for otherwise
$b_1$-$u$-$v$-$a_1$-$t''$-$u''$-$v''$ is an induced $P_7$.  Hence, by
symmetry, $a_1$ and $a_2$ have no neighbor in $Y''$.  Now $h$ has a
neighbor in $\{t'',u'',v'',w''\}$, for otherwise
$a_1$-$h$-$d$-$h''$-$t''$-$u''$-$v''$ is an induced $P_7$.  If $h$ has
two adjacent neighbors in $Y''$, then $h$ is complete to $Y''$, for
otherwise $d,h$ plus three consecutive vertices of $Y''$ induce a
bull; but then $\{h'',t'',u'',h,a_1\}$ induces a bull.  So we may
assume that $N_{C''}(h)=\{t'',v''\}$, for otherwise $u,v,a_1,h$ and
three consecutive vertices in $Y''$ induce a $P_7$.  But then
$u''$-$v''$-$h$-$d$-$h''$-$v$-$u$ is an induced $P_7$, a
contradiction.  Thus (\ref{u}) holds.

\medskip

Now, (\ref{case2}) follows from (\ref{hinz0}) and (\ref{u}).  This
completes the proof in Case~3.

\medskip

To conclude, we give the general outline of the algorithm to solve
MWSS in $K$.  For each type $q\in\{0,1,2\}$, we find a vertex $x$ such
that $K \setminus N(x)$ contains no $C_5$ of type~$q$.  We then solve
the MWSS in $K \setminus N(x)$ and in $K \setminus \{x\}$.  Since
every maximum weight stable set of $K$ either contains $x$ or not, the
best of these two solutions is a solution for the MWSS in $K$.  We
repeat this until there are no more $C_5$'s of this type.  More
formally:

\medskip

%%%%%%%%%%%%%%%
%% Algorithm %%
%%%%%%%%%%%%%%%

\noindent \textbf{(I)} Suppose that $K$ contains no $C_5$.  If $K$
also contains no $C_7$, then $K$ is perfect, so we can solve the MWSS
in $K$ by using the algorithm from~\cite{Penev}.  If $K$ contains a
$C_7$, then MWSS can be solved in time $O(|K|^3)$ as explained in
Case~1 of the proof.

\medskip

\noindent \textbf{(II)} Suppose that $K$ contains a $C_5$ of type~$2$
and no $C_5$ of type~$0$ or~$1$.  Let $h$ be a vertex of maximum score
as in Case~2 of the proof.  Then MWSS in $K$ can be solved by
successively solving the MWSS in (a) $G[K \setminus N(h)]$ and in (b)
$G[K \setminus \{h\}]$.

Step~(a) can be done as follows: If $G[K \setminus N(h)]$ contains no
$C_5$, then we are in (I).  If $G[K \setminus N(h)]$ contains a $C_5$
(of type~$2$), then by (\ref{xmax}) there is a vertex $a$ in this
$C_5$ such that $G[K \setminus (N(h)\cup N(a))]$ contains no $C_5$.
Hence we solve MWSS in (a1) $G[K \setminus (N(h)\cup N(a))]$ and in
(a2) $G[K \setminus (N(h)\cup \{a\})]$.  Step~(a1) can be done in
polynomial time by referring to (I).  Step~(a2) can be computed by
recursively calling Step~(a).  The number of recursive calls is
bounded by $|Z|$.

Step~(b) can be computed by recursively calling {(II)}.  After a
number of calls there is no longer any $C_5$ of type~$2$, so we are in
(I).  The number of recursive calls is bounded by $|H|$.

\medskip

\noindent \textbf{(III)} Suppose that $K$ contains a $C_5$ of type~$1$
and no $C_5$ of type~$0$.  Let $u$ be a vertex such that $K \setminus
N(u)$ has no $C_5$ of type~$1$, as in Claim~(\ref{u}).  Then MWSS in
$K$ can be solved by successively solving the MWSS in (a) $G[K
\setminus N(u)]$ and in (b) $G[K \setminus \{u\}]$.  Step~(a) can be
done in polynomial time by referring to (II) or (I).  Step~(b) can be
computed by recursively calling {(III)}.  After a number of calls
there is no longer any $C_5$ of type~$1$, so we are in {(II)} or (I).
The number of recursive calls is bounded by $|K|$.

\medskip

\noindent \textbf{(IV)} Suppose that $K$ contains a $C_5$ of type~$0$.
Let $T$ be the component of $G[Z]$ (unique by Claim~(\ref{z01})) that
contains a $C_5$.  Let $H_0=\{h\in H\mid$ $h$ has a neighbor in $T\}$.
Let $h$ be any vertex in $H_0$.  By (\ref{hinz0}) we know that $G[K
\setminus N(h)]$ contains no $C_5$ of type~$0$ or~$1$.  Then the MWSS
in $K$ can be solved by successively solving the MWSS in (a) $G[K
\setminus N(h)]$ and in (b) $G[K \setminus \{h\}]$.  Step~(a) can be
computed in polynomial time by calling {(II)} or (I).  Step~(b) can be
computed by recursively calling {(IV)}.  The number of recursive calls
is equal to $|H_0|$.  At the end of this step, the component $T$
becomes isolated because we have removed all vertices of $H_0$, but we
still need to solve MWSS in $T$.  This can be done as follows.
Consider any vertex $h\in H_0$.  By Claim~(\ref{z0}) every $C_5$ in
$T$ contains exactly two vertices from $N(h_0)\cap V(T)$, and these
two vertices are not adjacent.  Hence MWSS can be solved in $T$ using
the same technique as in (II) and the analogue of Claim~(\ref{xmax}).

\medskip

The total number of recursive calls is in $O(n)$ since there are three
different cycle types.  For each computation of MWSS in $K$, we end up calling
the algorithm in~\cite{Penev} which runs in $O(n^6)$.  Furthermore, at each
step we need to compute the list of all the cycles of length $5$, which takes
$O(n^5)$, but this is additive.  We need to run all the previous steps on every
connected component $K$ of the non-neighborhood of a fixed vertex of $V(G)$,
there are at most $n$ such components.  Finally, we repeat this for every vertex in
$V(G)$, so the overall complexity of our algorithm is $O(n^{9})$.  This
completes the proof of Theorem~\ref{thm:P7}.

\bigskip

One may wonder whether Claims~(\ref{xmax}) and (\ref{case2}) could be
subsumed by the following single claim: There is a vertex $x$ in $K$
such that $K\setminus N(x)$ contains no $C_5$ of any type.  Here is an
example showing that such a claim does not hold.  Let $Z$ have six
vertices $c_1,\ldots,c_5$ and $z$, such that $c_1,\ldots,c_5$ induce a
$C_5$ with edges $c_ic_{i+1}$ ($i\bmod 5$), and $z$ has no neighbor in
this $C_5$.  Let $H$ have five vertices $h_1,\ldots,h_5$ such that for
each $i$ we have $N_Z(h_i)=\{c_{i-1},c_{i+1},z\}$.  Let
$V(G)=\{c,d,h_1,\ldots,h_5,c_1,\ldots,c_5,z\}$.  It is a routine
matter to check that $G$ is $(P_7$, $K_3$)-free and that $K\setminus
N(x)$ contains a $C_5$ for every vertex $x\in K$.

\section{$(S_{1,2,3}$, bull)-free graphs}
\label{sec:S123}

\noindent{\bf Proof of Theorem~\ref{thm:S123}.} Let $G$ be a
$(S_{1,2,3},\mbox{bull})$-free graph, and let $w$ be a weight function
on the vertex set of $G$.  We proceed as in the proof of
Theorem~\ref{thm:P7}.  By Theorem~\ref{thm:LM}, we may assume that $G$
is prime, and by Lemmas~\ref{lem:wheel}---\ref{lem:nog12}, $G$
contains no wheel, no antiwheel, no umbrella, no $G_1$ and no $G_2$.
Let $c$ be any vertex of $G$, and let $K$ be an arbitrary component of
$G\setminus(\{c\}\cup N(c))$.  If $K$ is perfect, we can use the
algorithm from~\cite{Penev} to find a maximum weight stable set in
$K$.  Therefore let us assume that $K$ is not perfect.  By the Strong
Perfect Graph Theorem~\cite{CRST}, $K$ contains an odd hole or an odd
antihole.  In fact $K$ contains no antihole of length at least $6$,
for otherwise the union of such a subgraph and $c$ induces an
antiwheel.  So $K$ contains an odd hole.  We observe that:
\begin{equation}\label{ncs}
\longbox{If $C$ is a hole of length at least $5$ in $G$, and $x,y$ are
vertices in $V(G)\setminus V(C)$ such that $xy\in E(G)$ and $x$ has no
neighbor in $C$, then $N_C(y)$ is a stable set.}
\end{equation}
Proof: Let $C$ have length $\ell\ge 5$.  If $y$ has two consecutive
neighbors on $C$, then $y$ is complete to $V(C)$, for otherwise $x,y$
and three consecutive vertices of $C$ induce a bull; but then either
$V(C)\cup\{y\}$ induces a $k$-wheel, with $k \geq 6$ (if $\ell\ge 6$) or
$V(C)\cup\{x,y\}$ induces an umbrella (if $\ell=5$), a contradiction.  Thus
(\ref{ncs}) holds.

\medskip

Since $G$ is prime it is connected, so there is a neighbor $d$ of $c$
that has a neighbor in~$K$.  Let $H=N_K(d)$ and $Z=V(K)\setminus H$.
Now we claim that:
\begin{equation}\label{c75}
\longbox{$K$ contains no odd hole of length at least $7$.  Moreover,
every $C_5$ in $K$ contains one or two vertices of~$H$, and if it
contains two they are not adjacent.}
\end{equation}
Proof: Suppose that $K$ contains a hole $C$ of odd length $\ell\ge 5$,
with vertices $c_1, \ldots, c_\ell$ in order.  Since $G$ is prime, it
is connected, so there exists a path $p_0$-$p_1$-$\cdots$-$p_k$ with
$p_0\in V(C)$, $p_k=c$, and $k\ge 2$, and we choose a shortest such
path, so $p_1,\ldots,p_{k}\notin V(C)$ and $p_2,\ldots, p_k$ have no
neighbor in~$C$.  By~(\ref{ncs}) applied to $p_1,p_2$ and $C$, we know
that $N_C(p_1)$ is a stable set.  Since $\ell$ is odd, it follows that
there is an integer $i\in\{1,\ldots,\ell\}$ such that $p_1c_i\in E(G)$
and $p_1c_{i+1}, p_1c_{i+2}\notin E(G)$, and $p_1c_{i-1}\notin E(G)$,
say $i=1$.  If $\ell\ge 7$, then $p_1c_4\in E(G)$, for otherwise
$\{c_1,c_\ell, p_1,p_2,c_2,c_3,c_4\}$ induces an $S_{1,2,3}$; and then
$p_1c_5\notin E(G)$.  Then $p_1c_6\in E(G)$, for otherwise
$\{p_1,p_2,c_1,c_2,c_4,c_5,c_6\}$ induces an $S_{1,2,3}$.  But then
$\{p_1,p_2,c_6,c_5,c_1,c_2,c_3\}$ induces an $S_{1,2,3}$, a
contradiction.  This proves the first sentence of (\ref{c75}).  Now
$\ell=5$.  If $k\ge 3$, then $\{c_1,c_5,c_2,c_3,p_1,p_2,p_3\}$ induces
an $S_{1,2,3}$.  So $k=2$, and so $p_1=d$, and we already know that
$N_C(d)$ is equal to $\{c_1\}$ or $\{c_1,c_4\}$.  This proves the
second sentence of (\ref{c75}).  Thus (\ref{c75}) holds.

\medskip

For $q\in\{1,2\}$ we say that a $C_5$ in $K$ is \emph{of type~$q$} if
it contains exactly $q$ vertices from $H$.  By (\ref{c75}) and the
Strong Perfect Graph Theorem, $K$ contains a $C_5$, and every $C_5$ in
$K$ is of type~$1$ or~$2$.

For adjacent vertices $u,v$ in $Z$ we say that the edge $uv$ is
\emph{red} if there exists a $P_4$ $h'$-$u$-$v$-$h''$ for some
$h',h''\in H$.  For every vertex $h$ in $H$ we define its
\emph{score}, $sc(h)$, as the number of red edges that contain a
neighbor of $h$.

We choose a vertex $h_0\in H$ as follows: if there exists a red edge,
let $h_0$ be a vertex of maximum score in $H$; if there is no red
edge, let $h_0$ be any vertex in $H$ that has a neighbor in $Z$.  We
claim that:
\begin{equation}\label{h0}
\mbox{$K\setminus N(h_0)$ contains no $C_5$.}
\end{equation}
Proof: Suppose on the contrary that $K\setminus N(h_0)$ contains a
$C_5$ $C$.  First suppose that $C$ is of type~$1$.  So
$C=h$-$t$-$u$-$v$-$w$-$h$ for some $h\in H$ and $t,u,v,w\in Z$.  Let
$z$ be any neighbor of $h_0$ in $Z$.  By (\ref{ncs}) applied to
$h_0,z$ and $C$, we know that $N_{C}(z)$ is a stable set.  Suppose
that $zh\in E(G)$.  Then $zt, zw\notin E(G)$.  Then $z$ has a neighbor
in $\{u,v\}$, for otherwise $\{h,z,d,c,t,u,v\}$ induces an
$S_{1,2,3}$.  So, up to symmetry, $N_{C}(z)=\{h,u\}$.  But then
$\{u,t,v,w,z,h_0,d\}$ induces an $S_{1,2,3}$.  Therefore $z$ is not
adjacent to $h$.  Now $z$ has a neighbor in $\{t,u\}$, for otherwise
$\{d,c,h_0,z,h,t,u\}$ induces an $S_{1,2,3}$, and similarly $z$ has a
neighbor in $\{v,w\}$.  Since $N_{C}(z)$ is a stable set, we may
assume that $N_{C}(z)$ consists of $t$ plus one of $v,w$.  Then
$\{z,t,v,w,h_0,d,c\}$ induces an $S_{1,2,3}$, a contradiction.

\medskip

Now supppose that $C$ is of type~$2$.  So
$C=t$-$h_1$-$a$-$b$-$h_2$-$t$, for some $h_1,h_2\in H$ and $a,b,t\in
Z$.  Clearly $h_0\notin\{h_1,h_2\}$.  Note that $ab$ is a red edge.
There must be a red edge $y_1z_1$ (with $y_1,z_1\in Z$) that is
counted in $sc(h_0)$ and not in $sc(h_1)$, for otherwise we have
$sc(h_1)\ge sc(h_0)+1$ (because of $ab$), which contradicts the choice
of $h_0$.  So $h_1$ has no neighbor in $\{y_1,z_1\}$.  We may assume
that $h_0y_1\in E(G)$.  Let $h'$-$y_1$-$z_1$-$h''$ be a $P_4$ with
$h',h''\in H$.  If $h_0z_1\in E(G)$, then $h_0h'\notin E(G)$, for
otherwise $\{c,d,h',h_0,z_1\}$ induces a bull; and similarly
$h_0h''\notin E(G)$; but then $\{h',y_1,h_0,z_1,h''\}$ induces a bull.
Hence $h_0z_1\notin E(G)$.  Clearly $a\notin\{y_1,z_1\}$.  If $a$ has
no neighbor in $\{y_1,z_1\}$, then $\{d,c,h_1,a,h_0,y_1,z_1\}$ induces
an $S_{1,2,3}$; while if $a$ is complete to $\{y_1,z_1\}$, then
$\{h_0,y_1,z_1,a,h_1\}$ induces a bull.  Hence $a$ has exactly one
neighbor in $\{y_1,z_1\}$.
\\
Suppose that $a$ is adjacent to $y_1$ and not to $z_1$.  Then
$y_1b\notin E(G)$, for otherwise $\{h_0,y_1,b,a,$ $h_1\}$ induces a
bull; and $z_1b\in E(G)$, for otherwise $\{y_1,z_1,a,b,h_0,d,$ $c\}$
induces an $S_{1,2,3}$.  If $ty_1\in E(G)$, then $h_2y_1\notin E(G)$,
for otherwise $\{a,y_1,t,h_2,$ $d\}$ induces a bull; but then
$\{y_1,t,z_1,b,h_0,d,c\}$ induces an $S_{1,2,3}$.  Hence $ty_1\notin
E(G)$.  Then $tz_1\notin E(G)$, for otherwise
$\{y_1,a,z_1,t,h_0,d,c\}$ induces an $S_{1,2,3}$.  But now
$\{h_1,t,d,c,a,y_1,z_1\}$ induces an $S_{1,2,3}$, a contradiction.
\\
Therefore $a$ is adjacent to $z_1$ and not to $y_1$.  If $z_1b\in
E(G)$, then $z_1h_2\in E(G)$, for otherwise $\{h_1,a,z_1,b,h_2\}$
induces a bull; but then $\{c,d,h_1,a,b,z_1,h_2\}$ induces a $G_2$.
Hence $z_1b\notin E(G)$.  Then $y_1b\in E(G)$, for otherwise
$\{a,b,z_1,y_1,h_1,d,c\}$ induces an $S_{1,2,3}$; and $ty_1\in E(G)$,
for otherwise $\{h_1,t,d,c,a,b,y_1\}$ induces an $S_{1,2,3}$.  But
then $\{y_1,t,b,a,h_0,d,c\}$ induces an $S_{1,2,3}$.  Thus (\ref{h0})
holds.

\medskip

Finally we give the outline of the algorithm to solve MWSS in $K$.

\begin{itemize}

\item[(I)] Suppose that $K$ contains no $C_5$.  Then $K$ is perfect,
so we can compute the MWSS in $K$ by using the algorithm
from~\cite{Penev}.

\item[(II)] Suppose that $K$ contains a $C_5$.  Let $h_0$ be a vertex
of $H$ as in (\ref{h0}).  Then MWSS in $K$ can be solved by
successively solving MWSS in (a) $K\setminus{N(h_0)}$ and in (b) $K
\setminus \{h_0\}$.  Step~(a) can be done in polynomial time by
referring to (I) since $K\setminus{N(h_0)}$ is perfect by (\ref{c75})
and (\ref{h0}).  Step~(b) can be computed by recursively calling (II).
The number of recursive calls is bounded by~$|H|$.
\end{itemize}

For each computation of MWSS in $K$, we end up calling the algorithm
in~\cite{Penev} which runs in $O(n^6)$.  Furthermore, at each step we need to
compute the list of all the cycles of length $5$, which takes $O(n^5)$, but
this is additive. 
We need to run all the previous steps on every connected component $K$ of the
non-neighborhood of a fixed vertex of $V(G)$, there are at most $n$ such components.
Finally, we repeat this for every vertex in $V(G)$, so the
overall complexity of our algorithm is $O(n^{8})$. 
This completes the proof of Theorem~\ref{thm:S123}.

\section{Concluding remarks}

The technique used here is essentially that which was developed by
Brandst\"adt and Mosca in \cite{BM}.  The new results that it has
enabled us to establish are generalizations of results \cite{KM2,MP}
that were obtained earlier using different techniques.  It is not
clear to us if the same ideas can be used to solve MWSS in other
classes of graphs, such as ($S_{2,2,2}$, bull)-free graphs or ($P_8$,
bull)-free graphs or similar classes (not necessarily subclasses of
bull-free graphs).  These may be interesting open problems.

\clearpage
% \small

\end{document}